\documentclass[12pt, a4paper, twoside]{article}
\usepackage{amsmath,amsfonts,amssymb,amsthm,amscd}
\usepackage{epsfig}
\numberwithin{equation}{section}
\newtheorem{theorem}{Theorem}[section]

\newtheorem{lemma}[theorem]{Lemma}
\newtheorem{claim}[theorem]{Claim}

\newtheorem{corollary}[theorem]{Corollary}
\newtheorem{observation}[theorem]{Observation}

\newtheorem*{remark*}{Remark}

\newcommand{\Fcal}{\mathcal{F}}
\newcommand{\Lcal}{\mathcal{L}}
\newcommand{\Pcal}{\mathcal{P}}
\newcommand{\Rcal}{\mathcal{R}}
\newcommand{\Ucal}{\mathcal{U}}

\def\inst#1{$^{#1}$}
\newcommand{\theforbmatrix}{\bigl( \begin{smallmatrix}1 & 1 \\ 1& 1\end{smallmatrix}\bigr)}

\begin{document}

\title{Maximum size of reverse-free sets of permutations
\thanks{
This research was supported by the Czech Science Foundation under the contract no. 201/09/H057 
and by grant SVV-2012-265313 (Discrete Methods and Algorithms).
} %end of thanks
} %end of title

\author{Josef Cibulka\inst{1}
} %end of author

\date{}

\maketitle

\begin{center}
{\footnotesize
\inst{1}
Department of Applied Mathematics and Institute for Theoretical Computer Science, \\
Charles University, Faculty of Mathematics and Physics, \\
Malostransk\'e n\'am.~25, 118~00 Praha 1, Czech Republic; \\
\texttt{cibulka@kam.mff.cuni.cz}
}
\end{center}  %end of institute

\begin{abstract}
Two words have a reverse if they have the same pair of distinct letters on the same pair of positions, but in reversed order.
A set of words no two of which have a reverse is said to be reverse-free.
Let $F(n,k)$ be the maximum size of a reverse-free set of words from $[n]^k$ where no letter repeats within a word.
We show the following lower and upper bounds in the case $n \ge k$: $F(n,k) \in n^k k^{-k/2 + O\left(k /\log k \right)}$.
As a consequence of the lower bound, a set of $n$-permutations each two having a reverse has size at most $n^{n/2 + O \left(n/\log n \right)}$.
\end{abstract}

%\subsection*{Keywords}
%set of permutations
%\end{keyword}

\section{Introduction}

Let $[n]$ be the set of integers from $1$ to $n$.
A \emph{word} $w$ of length $k$ over the alphabet $A$ is a sequence $w_1, \dots, w_k$
of elements from $A$.
The set of all words of length $k$ over $[n]$ is $[n]^k$.
A word is \emph{repetition-free} if it contains at most one occurrence of each symbol.
The set of all repetition-free words of length $k$ over $[n]$ is $[n]_{(k)}$.
Notice that when $n = k$, the set $[n]_{(n)}$ is the set $S_n$ of permutations on $n$ elements .
A \emph{code} $\Fcal$ of length $k$ is a subset of $[n]^k$.
The \emph{size} of $\Fcal$ is the number of words in $\Fcal$.
Codes are usually defined to be sets of words that in some sense significantly differ from each other
in order to be distinguishable when transmitted over a noisy channel.
We study reverse-free codes introduced by F{\"u}redi, Kantor, Monti and Sinaimeri~\cite{FKMS}.
Two words $w$ and $x$ have a \emph{reverse} if for some pair $(i,j)$ of positions, 
we have $w_i \neq w_j$, $w_i = x_j$ and $w_j = x_i$.
If $w$ and $x$ do not have a reverse, they are \emph{reverse-free}.
A code is \emph{reverse-free} if its words are pairwise reverse-free.
Let $\overline{F}(n,k)$ be the size of the largest reverse-free code over $[n]$ of length $k$.
Let $F(n,k)$ be the size of the largest reverse-free code over $[n]$ of length $k$ containing only
repetition-free words.
Let
\[
f(k)
\stackrel{def}{=} \lim_{n \rightarrow \infty} \frac{F(n,k)}{k! \binom{n}{k}}.
\]
The limit exists for every $k \ge 1$~\cite{FKMS}.
We will use the following equivalent definitions of the limit:
\[
f(k)
= \lim_{n \rightarrow \infty} \frac{F(n,k)}{n^k}
= \lim_{n \rightarrow \infty} \frac{\overline{F}(n,k)}{n^k}.
\]
The first equality follows from the fact that $\lim_{n \rightarrow \infty} \binom{n}{k} k! n^{-k} = 1$.
The second equality is a consequence of the observation that for every fixed $k$,
we have $F(n,k) \le \overline{F}(n,k) \le F(n,k) + O(n^{k-1})$~\cite{FKMS}.
The only exact values of the limit known are $f(1) = 1$, $f(2) = 1/2$ and $f(3) = 5/24$~\cite{FKMS}.

We tighten the bounds on the maximum size of reverse-free codes 
of length greater or equal to the size of the alphabet.
\begin{theorem}
\label{thm:main}
For every $n \ge k$, we have
\[
n^k k^{-k/2 - O\left(k/\log k \right)} \le
F(n,k) \le \overline{F}(n,k) \le
n^k k^{-k/2 + O\left(k /\log k \right)}.
\]
\end{theorem}

The first inequality is proven in Section~\ref{sec:lb} as Corollary~\ref{cor:lbGeneral}
and the last inequality is proven as Claim~\ref{cl:ub} in Section~\ref{sec:ub}.
As an immediate consequence, we obtain the following bounds for permutation codes:
\[
n^{n/2 - O \left(n/\log n \right)} \le
F(n,n) \le \overline{F}(n,n) \le
n^{n/2 + O \left(n /\log n \right)}
\]
and for the limit for codes of fixed length $k$:
\[
f(k) \in k^{-k/2 + O(k/\log k)}.
\]

A set of words is \emph{full of flips} if each two words from the set have a reverse.
Let $\overline{G}(n,k)$ be the size of the largest code full of flips with elements in $[n]^{k}$.
Let $G(n,k)$ be the size of the largest code full of flips with elements in $[n]_{(k)}$.
By $G(n,n)F(n,n) \le n!$~\cite{FKMS}, we obtain the following corollary.
\begin{corollary}
The size of a set of permutations full of flips is at most
\[
G(n,n) \le n^{n/2 + O \left(n/\log n \right)}.
% G(n,n) \le n^{\frac{n}{2} + O \left(\frac{n}{\log n} \right)}
\]
\end{corollary}

A position of an entry of a matrix is represented by a pair $(r,c)$ of the row number $r$ and the column number $c$.
A $\{0,1\}$-matrix is a matrix whose each entry is either $0$ or $1$.
Every matrix in this paper is a $\{0,1\}$-matrix, even when it is not explicitly mentioned.

All logarithms in this paper are of base $2$.

\section{Lower Bound}
\label{sec:lb}
A submatrix of a matrix $B$ is a matrix that can be obtained from $B$ by the removal of some columns and rows.
An $m\times n$ $\{0,1\}$-matrix $B$ \emph{contains} a $k\times l$ $\{0,1\}$-matrix $Q$ if $B$ has a $k\times l$
submatrix $T$ that can be obtained from $Q$ by changing some (possibly none) $0$-entries to $1$-entries.
Otherwise $B$ \emph{avoids} $Q$.

F\"{u}redi and Hajnal~\cite{FurediHajnal} studied the following problems from the extremal theory of
$\{0,1\}$-matrices. Given a matrix $Q$ (the \emph{forbidden matrix}), what is the maximum number
of $1$-entries in an $n\times n$ matrix that avoids $Q$?

We restrict our attention on forbidding the $2 \times 2$ matrix with each entry equal to $1$
and we call this matrix $S$.
Maximizing the number of $1$'s in a matrix avoiding $S$ is closely related to 
maximizing the number of edges in an $n$-vertex graph without a $4$-cycle as a subgraph~\cite{FurediHajnal}.
The maximum number of edges in a $4$-cycle-free graph
is known precisely for infinitely many values of $n$~\cite{Furedi96}.
We will use a classical construction of a bipartite $4$-cycle-free graph 
(see for example the book of Matou\v{s}ek and Ne\v{s}et\v{r}il~\cite{MNBook}).
We reproduce the construction here in the matrix setting since we need some of its additional properties.

%Asymptotically tight bounds on the problem on graphs were found by K\'ovari, S\'os and Tur\'an~\cite{KST}.
%These bounds readily translate to the matrix problem as observed by F\"{u}redi and Hajnal~\cite{FurediHajnal}.

The construction of a matrix avoiding $S$ builds the matrix using a finite projective plane.
Let $X$ be a finite set and let $\Lcal$ be a family of subsets of $X$.
The set system $(X, \Lcal)$ is a \emph{finite projective plane} if
\begin{enumerate}
\item[(P0)] There is a $4$-tuple $F$ of elements of $X$ such that $|F \cap L| \le 2$ for every $L \in \Lcal$.
\item[(P1)] For every $L_1, L_2 \in \Lcal$, $|L_1 \cap L_2| = 1$.
\item[(P2)] For every $x,y \in X$ there exists exactly one $L \in \Lcal$ containing both $x$ and $y$.
\end{enumerate}

For every finite projective plane, we can find a number $r$, called the \emph{order of the projective plane},
satisfying:
\begin{enumerate}
\item[(P3)] For every $L \in \Lcal$, $|L| = r+1$.
\item[(P4)] Every $x \in X$ is contained in exactly $r+1$ sets from $\Lcal$.
\item[(P5)] We have $|X| = |\Lcal| = r^2 + r + 1$. This value is the \emph{size} of the projective plane.
\end{enumerate}

It is known that for every number $r$ that is a power of a prime number,
we can find a finite projective plane of order $r$~\cite{MNBook}.

\begin{claim}
\label{cl:lbprime}
If $n$ is of the form $r^2 + r + 1$, where $r$ is a power of a prime,
then 
\[
F(n,n) \ge n^{n/2 - O(n/\log n)}. 
\]
\end{claim}

\begin{proof}
We fix a projective plane $(X, \Lcal)$ of size $n$.
We order the elements of $X$ and the sets of $\Lcal$ arbitrarily.
The \emph{incidence matrix} of a finite projective plane of size $n$
is the $n \times n$ matrix $A$ with $1$ on position $(i,j)$
exactly if the $i$-th set of $\Lcal$ contains the $j$-th element of $X$.
Let $A$ be the incidence matrix of $(X, \Lcal)$.
%The matrix $A$ avoids the matrix $S$ by (P1)
%and has exactly $r+1 \ge \sqrt{n}$ $1$-entries in every row and in every column.

An \emph{$n$-permutation matrix} is an $n \times n$ matrix with exactly one $1$-entry in every column and every row.
An \emph{$n$-permutation} is a permutation on $n$ elements.
The following is a bijection between the set of $n$-permutations and the set of $n$-permutation matrices.
A permutation $\pi$ is matched with the matrix $P$ with $1$ on position $(i,j)$ exactly if $\pi_i = j$.
Let $\Pcal$ be the set of $n$-permutation matrices contained in $A$ 
and let $\Pi$ be the set of $n$-permutations matched to the matrices from $\Pcal$.
By (P3) and (P4), $A$ has exactly $r+1$ $1$'s in every row and every column.
Thus by the van der Waerden conjecture proved independently by Falikman~\cite{Falikman} and Egorychev~\cite{Egorychev},
%the number of $n$-permutation matrices contained in $A$ is at least
\[
|\Pi| = |\Pcal| 
\ge \left(\frac{r+1}{n}\right)^n n!
\ge \left(\frac{r+1}{e}\right)^n
\ge \left(\frac{n^{1/2}}{e}\right)^n
\ge n^{n/2 - O(n/\log n)}.
\]

We claim that the set $\Pi$ is pairwise reverse-free.
For contradiction, we take $\pi \in \Pi$ and $\rho \in \Pi$ with a reverse on positions $i$ and $j$.
That is, $\pi_i = k$, $\pi_j = l$, $\rho_i = l$, $\rho_j = k$ for some $k$ and $l$.
Since $P_{\pi}$ and $P_{\rho}$ are contained in $A$,
this implies that $A$ contains the matrix $S$ on rows $i,j$ and columns $k$ and $l$; a contradiction with (P1).
\end{proof}

By the prime number theorem,
the gaps between two consecutive prime numbers in proportion to the primes tend to zero.
There has been a significant progress in tightening the gap between two consecutive primes.
Most recent is the following result of Baker, Harman and Pintz~\cite{BHP}.

\begin{theorem}[Baker, Harman, Pintz, 2001]
For every large enough $n$, the interval $[n-n^{0.525}, n]$ contains a prime number.
\end{theorem}

\begin{lemma}
\label{lem:lbperm}
For every $n$,
\[
F(n,n) \ge n^{n/2 - O \left(n/\log n \right)}.
\]
\end{lemma}
\begin{proof}
For an arbitrary $n$ we take the largest $n'$ smaller than $n$
and expressible as $p^2 + p + 1$ for some prime $p$.
The interval $[n^{1/2} - 1 - n^{0.525/2}, n^{1/2} - 1]$ contains a suitable prime number $p$.
Thus
\begin{align*}
p &\ge n^{1/2} - 1 - n^{0.525/2}  \quad \text{and} \\
n' &\ge n-O(n^{1.525/2}).
\end{align*}
We take the set $\Pi'$ of $(n')^{n'/2 - O(n'/\log n')}$
pairwise reverse-free $n'$-permutations from Claim~\ref{cl:lbprime}.
We append the sequence $(n'+1, n'+2, \dots, n)$ to the end of each $\pi' \in \Pi'$.
Let the resulting set of $n$-permutations be $\Pi$.
The set $\Pi$ of permutations is pairwise reverse-free and has size at least $n^{n/2 - O(n/\log n)}$.
\end{proof}

\begin{lemma}
\label{lem:lbGeneral}
For every $n \ge k$,
\[
F(n,k) \ge \left\lfloor \frac{n}{k}\right\rfloor^k  F(k,k).
\]
\end{lemma}
\begin{proof}
Let $\Pi$ be a reverse-free set of $k$-permutations of size $F(k,k)$.
Given a word $u = (u_1,u_2, \dots, u_k) \in [n]_{(k)}$,
we call the word $(u_1 \bmod k, u_2 \bmod k, \dots, u_k \bmod k)$ the \emph{compression} of $u$.
Let $\Fcal$ be a set of all the words in $[n]_{(k)}$ whose compression is in $\Pi$.
The size of $\Fcal$ is at least $\left\lfloor \frac{n}{k} \right\rfloor ^k |\Pi|$.
It remains to show that $\Fcal$ is reverse-free.
For contradiction, assume that some pair of words $(u_1, \dots, u_k)$ and $(v_1, \dots, v_k)$
has a reverse on the pair $(i,j)$ of positions.
That is, $u_i = v_j$ and $u_j = v_i$ and,
in particular, $u_i \bmod k = v_j \bmod k$ and $u_j \bmod k = v_i \bmod k$.
Because the compression of $u$ is a permutation, $u_i \bmod k \neq u_j \bmod k$.
This is a contradiction, because the compressions of $u$ and $v$ are in the reverse-free set $\Pi$.
\end{proof}

\begin{corollary}
\label{cor:lbGeneral}
For every $n \ge k$,
\[
F(n,k) \ge n^k k^{-k/2 - O\left(k/\log k \right)}.
\]
\end{corollary}
\begin{proof}
Since $n \ge k$, we have $\lfloor n/k \rfloor \ge n/(2k)$.
Therefore, by Lemmas~\ref{lem:lbperm}~and~\ref{lem:lbGeneral}
\[
F(n,k) \ge 
\left(\frac{n}{2k}\right)^k F(k,k) 
\ge \frac{n^k}{(2k)^k} k^{k/2 - O(k/\log k)} 
%\ge \frac{n^k}{2^k} k^{-k/2 - O(k/\log k} 
\ge n^k k^{-k/2 - O(k/\log k)}.
\]

\end{proof}

\section{Upper Bound}
\label{sec:ub}
We use a result claiming that a matrix with many $1$-entries contains many occurrences of the matrix $S = \theforbmatrix$.
This corresponds to counting the occurrences of $K_{2,2}$ in a bipartite graph.
Erd\H{o}s and Simonovits~\cite{ErdosSimonovits} proved that an arbitrary graph $G$ with $e(G)$ edges and $v(G)$ vertices 
contains at least $e(G)^4/(2v(G)^4) - e(G)^2/(2v(G))$ copies of $K_{2,2}$.
Sidorenko~\cite{Sidorenko} proves a general result that also gives a lower bound on the number of occurrences of $K_{2,2}$ 
(and several other bipartite graphs) in a graph with many edges. 
It follows from~\cite[Condition B]{Sidorenko} that a bipartite graph with parts of size $n$ and $k$ 
contains $e(G)^4/(4 n^2 k^2) - O(nk^2 + n^2k)$ copies of $K_{2,2}$.

Neither of these results is applicable in cases when $k$ is much smaller 
than $n$ and the number of edges is of the order $O(nk^{1/2})$.
Thus, we follow the approach used in the above mentioned papers 
to prove the following lemma which gives a more precise bound in such cases.
This approach appeared already in 1964 
in a proof of a similar result of Erd\H{o}s and Moon~\cite{ErdosMoon}.

\begin{lemma}
\label{l:countS}
Let $k$ and $n$ be integers such that $n \ge k \ge 1$ 
and let $m$ be a real number from the closed interval $[1, k^{1/2}]$.
Let $A$ be a $k \times n$ $\{0,1\}$-matrix with at least $mnk^{1/2}$ $1$-entries.
The number of occurrences of $S$ in $A$ is at least
\[
\frac{n^2(m^2-1)^2}{4} - m^3 n k^{1/2}.
\]
\end{lemma}
\begin{proof}
If $k = 1$, then $A$ has at most $n$ $1$'s and so $m=1$ and the claim is trivially satisfied.
So we assume that $k \ge 2$.
We also assume that $A$ has no empty rows. 
If $A$ has empty rows, we remove them and use the claim for the matrix with no empty rows.
Since the removal increases $m$, and does not change $mnk^{1/2}$ and $n$, 
we obtain at least the required number of occurrences of $S$.

We first count the number $q$ of pairs of $1$'s that are in the same row.
Let $d_i$ be the number of $1$'s in the $i$-th row of $A$.
We have 
\[
q = \sum_{\{i: d_i \ge 2 \}} \binom{d_i}{2} \ge \sum_{i=1}^{k} \frac{(d_i-1)^2}{2}. 
\]
Let $d$ be the average number of $1$'s in a row, that is,
\[
d \stackrel{\text{def}}{=} \frac{\sum_{i = 1}^{k} d_i}{k} \ge \frac{mnk^{1/2}}{k} = mn k^{-1/2}.
\]
By the convexity of the function $f(x) = (x-1)^2/2$, we have
\[
q 
\ge k \frac{(d-1)^2}{2}
\ge \frac{k}{2} \left(mnk^{-1/2}-1\right)^2
\ge \frac{(mn-k^{1/2})^2}{2}
\ge \frac{m^2 n^2}{2} - m n k^{1/2}.
\]

Let $r_{i,j}$ be the number of rows that have a $1$-entry in columns $i$ and $j$.
Let $\Rcal$ be the set of pairs $\{i,j\}$ of column indices satisfying $1 \le i < j \le n$ and $r_{i,j} \ge 2$.

First, we consider the case $q \le n^2/2$. 
From the estimate $q \ge (mn-k^{1/2})^2/2$, we obtain $m \le 1+k^{1/2}/n$. 
Therefore $n^2(m^2-1)^2 \le (m+1)^2k \le 4m^2k$ 
and the result holds trivially because $m^2k \le m^3nk^{1/2}$.

Now, we assume $q > n^2/2$, which implies $|\Rcal| > 0$.
By double counting,
\[
q 
= \sum_{1 \le i < j \le n} r_{i,j} 
\le \sum_{\{i,j\}\in \Rcal} r_{i,j} + \binom{n}{2} - |\Rcal|.
\]
Let
\[
r 
\stackrel{\text{def}}{=} \frac{\sum_{\{i,j\}\in \Rcal} r_{i,j}}{|\Rcal|}
\ge \frac{q-\left(\binom{n}{2} - |\Rcal| \right)}{|\Rcal|}
= \frac{q-\binom{n}{2}}{|\Rcal|} + 1.
\]
Let $s$ be the number of occurrences of $S$ in $A$,
that is, $s = \sum_{\{i,j\} \in \Rcal} \binom{r_{i,j}}{2}$.
By the convexity of $f(x) = (x-1)^2/2$ and since $r > 1$, we have
\begin{align*}
s 
& \ge |\Rcal| \frac{(r-1)^2}{2}
\ge \frac{|\Rcal|}{2} \left(\frac{q - \binom{n}{2}}{|\Rcal|}\right)^2 \\
& \ge \frac{\left(m^2 n^2 / 2 - m n k^{1/2} - n^2/2 \right)^2}{2|\Rcal|} \\
& \ge \frac{\left(n^2 (m^2 - 1) / 2 - m n k^{1/2}\right)^2}{n^2} \\
& \ge \frac{n^2(m^2 - 1)^2}{4} - m^3 n k^{1/2}.
\end{align*}

\end{proof}

We first give some definitions and outline the proof of the upper bound in Theorem~\ref{thm:main}
without mentioning precise values used.
We use a modification of a method of Raz~\cite{Raz00}, that was used for proving upper bounds 
in another extremal problem on sets of permutations~\cite{Raz00,CK12}.

A \emph{$k \times n$ word matrix} is a $k \times n$ matrix with exactly one $1$-entry in every row.
A \emph{$k \times n$ word} $u$ is a sequence $u_1, u_2, \dots, u_k$ of $k$ letters from the alphabet $[n]$.
The following is a bijection between the set of $k \times n$ words and the set of $k \times n$ word matrices.
A word $u$ is matched with the matrix $U$ having $1$ on position $(i,j)$ exactly if $u_i = j$.
A set $\Ucal$ of $k \times n$ word matrices is \emph{reverse-free} 
if the set of corresponding words is reverse-free.
%if there is a pair $\{r_1,r_2\}$ of rows
%and a pair $\{c_1, c_2\}$ of columns such that $U$ has $1$-entries

Given a set $\Ucal$ of $k \times n$ word matrices,
we let the \emph{overall matrix} $A_{\Ucal}$ be the $k \times n$ matrix
having $1$-entries on those positions where at least one matrix of $\Ucal$ has a $1$-entry.
The basic idea is to design a procedure that shrinks the set $\Ucal$
in order to decrease the number of $1$'s in the overall matrix.
When the overall matrix has few $1$'s, we use a trivial estimate on the size of what remained in $\Ucal$.
By analyzing the procedure, we then deduce that the original size of $\Ucal$ was small.

The shrinking procedure uses the result of Lemma~\ref{l:countS} applied on the overall matrix.
Assume that the overall matrix contains $S$ on the intersection of rows $r_1$ and $r_2$ and columns $c_1$ and $c_2$.
Let an \emph{avoided pair} be a pair of $1$-entries of the overall matrix
that do not appear together in any matrix in $\Ucal$.
By the reverse-free property of $\Ucal$, we know that at least one of the two pairs $\{(r_1,c_1),(r_2,c_2)\}$
and $\{(r_1,c_2),(r_2,c_1)\}$ is avoided.
When the overall matrix contains many occurrences of $S$, we find a $1$-entry $(r,c)$
occurring in many avoided pairs.
If the $1$-entry $(r,c)$ is not present in enough matrices from $\Ucal$,
we remove from $\Ucal$ all the matrices containing $(r,c)$,
thus removing $(r,c)$ from the overall matrix.
Otherwise, we keep only the matrices that contain $(r,c)$,
thus removing all the matrices containing any of the $1$-entries
that appear in some avoided pair together with $(r,c)$.

Given a reverse-free set $\Ucal$ of $k \times n$ word matrices, 
let $A_{\Ucal}$ be the overall matrix of $\Ucal$.
Let the \emph{weight} $|A_{\Ucal}|$ of the overall matrix be the number of its $1$-entries.
The \emph{density} of the overall matrix is $m_{\Ucal} = |A_{\Ucal}|/(nk^{1/2})$.
The $1$-entry of the overall matrix $A_{\Ucal}$ on the position $(r,c)$ is \emph{light} if the number
of matrices $U \in \Ucal$ having $1$ on position $(r,c)$ is at most $|\Ucal|/n$.
Let the \emph{emptiness} $z_{\Ucal}$ of $\Ucal$ be the number of rows of $A_{\Ucal}$ with at most one $1$-entry.

\begin{observation}
\label{obs:light}
Let $\Ucal$ be a reverse-free set of $k \times n$ word matrices such that $A_{\Ucal}$ has a light $1$-entry.
Then the set $\Ucal'$ of word matrices of $\Ucal$ not containing the light $1$-entry satisfies
\begin{align*}
|\Ucal'| & \ge \left(1-\frac{1}{n}\right) |\Ucal|, \\
|A_{\Ucal'}| & \le |A_{\Ucal}|-1 \quad \text{and} \\
z_{\Ucal'} & \ge z_{\Ucal}.
\end{align*}
\end{observation}

Let $n_0$ be a constant such that for every $n \ge n_0$, $k\le n$ and $m \geq 5$,
every matrix $A$ with $mnk^{1/2}$ $1$'s contains $n^2m^4 / 5$ occurrences of $S$.
The existence of $n_0$ follows from Lemma~\ref{l:countS}.

\begin{claim}
\label{cl:nolight}
Let $n \ge n_0$ and let $k \le n$.
Let $\Ucal$ be a reverse-free set of $k \times n$ word matrices 
with $m_{\Ucal} \ge 5$ and such that $A_{\Ucal}$ has no light $1$-entry.
Then there exists a set $\Ucal' \subset \Ucal$ satisfying
\begin{align*}
|\Ucal'| & \ge \frac{|\Ucal|}{n}, \\
|A_{\Ucal'}| & \le |A_{\Ucal}|-\frac{2n m_{\Ucal}^3}{5k^{1/2}} \quad \text{and} \\
z_{\Ucal'} & \ge z_{\Ucal}+1.
\end{align*}
\end{claim}
\begin{proof}
The overall matrix $A_{\Ucal}$ contains $n^2m_{\Ucal}^4 / 5$ occurrences of $S$.
So at least $n^2m_{\Ucal}^4 / 5$ pairs of $1$-entries of $A_{\Ucal}$ are avoided.
Thus there is a $1$-entry of $A_{\Ucal}$ such that the number of avoided pairs containing this $1$-entry is at least
\[
\frac{2n^2m_{\Ucal}^4}{5 |A_{\Ucal}|}
= \frac{2n^2m_{\Ucal}^4}{5 nk^{1/2} m_{\Ucal}}
= \frac{2nm_{\Ucal}^3}{5k^{1/2}}.
\]
Let $(r,c)$ be the position of this $1$-entry.
Let $\Ucal'$ be the set of those matrices from $\Ucal$ that have $1$ at position $(r,c)$.
We consider a position $(r',c')$ such that $\{(r,c), (r',c')\}$ is an avoided pair.
Every matrix $U' \in \Ucal'$ has $0$ at position $(r',c')$.
So also $A_{\Ucal'}$ has $0$ at position $(r',c')$.
Therefore $|A_{\Ucal'}| \le |A_{\Ucal}|-2nm_{\Ucal}^3 / (5k^{1/2})$.
Because $(r,c)$ is not a light $1$-entry, $|\Ucal'|  \ge |\Ucal|/n$.
Since $\Ucal'$ contains only word matrices, the matrix $A_{\Ucal'}$ contains only one $1$-entry in row $r$.
On the other hand, the $1$-entry at position $(r,c)$ is contained 
in at least one occurrence of $S$ in $A_{\Ucal}$,
so $A_{\Ucal}$ contains more than one $1$-entry in row $r$.
Thus $z_{\Ucal'} \ge z_{\Ucal}+1$.
\end{proof}

\begin{claim}
\label{cl:ub}
Let $\Ucal$ be a set of $n \times k$ word matrices, where $n \ge k$. If $\Ucal$ is reverse-free, then
\[
|\Ucal| \le n^k k^{-k/2 + O(k/\log k)}
\]
\end{claim}
\begin{proof}
We first consider the case that the density $m_{\Ucal}$ of the overall matrix is smaller than $10$.
Since the number of $k \times n$ word matrices contained in $A_{\Ucal}$ is maximized
when each of its rows has the same number of $1$-entries, we obtain
\[
|\Ucal| \le (10nk^{-1/2})^k
\]
and the result holds.

Otherwise, we apply the following procedure on $\Ucal$.
We proceed in several steps.
Let $\Ucal_i \subset \Ucal$ be the set of word matrices before the step $i$.
Let $\Ucal_1 = \Ucal$.
If the overall matrix at the beginning of the step $i$ has a light $1$-entry,
we obtain $\Ucal_{i+1}$ from $\Ucal_{i}$ by applying Observation~\ref{obs:light};
otherwise by applying Claim~\ref{cl:nolight}.
Let $m_i \stackrel{def}{=} m_{\Ucal_i}$ and $A_i \stackrel{def}{=} A_{\Ucal_i}$.
\emph{Light steps} are the steps when Observation~\ref{obs:light} is applied and
\emph{heavy steps} are the remaining ones.
The steps are further grouped into \emph{phases}.
Phase $1$ starts with step $p_1 = 1$.
For every $j \ge 2$, phase $j$ starts with step $p_j$ chosen as the smallest index such that
$m_{p_j} \le m_{p_{j-1}}/2$.
The last phase is the first phase $\ell$ that decreases the density of the overall matrix below $10$.
So at the beginning of the last phase, we have
\[
m_{p_{\ell}} \ge 10.
%m_{p_{\ell+1}-1} \ge \frac{m_{p_{\ell}}}{2} \ge 5.
\]

Because each light step decreases the number of $1$'s in the overall matrix by $1$,
only at most $nk$ light steps are done during the whole procedure.

It remains to count the heavy steps.
At the beginning of a heavy step $i$ of phase $j$, we have
\[
%|A_{p_{j}}| \ge |A_{i}| &\ge |A_{p_{j}}|/2 \\
m_i \ge m_{p_{j}}/2.
\]
By Claim~\ref{cl:nolight}, the heavy step decreases the number of $1$-entries in the overall matrix by
\[
|A_{i}| - |A_{i+1}| 
\ge \frac{2nm_{i}^3}{5k^{1/2}} \ge \frac{nm_{p_{j}}^3}{20 k^{1/2}}.
\]
Since the phase $j$ ends at the moment when at least $|A_{p_{j}}|/2$ $1$-entries are removed, 
the number of heavy steps of phase $j$ is at most
\[
\left\lceil \frac{nk^{1/2}m_{p_{j}}/2}{nm_{p_{j}}^3 / (20 k^{1/2})} \right\rceil
= \left\lceil \frac{10k}{m_{p_{j}}^2} \right\rceil.
\]
Each phase shrinks the weight of the overall matrix by a factor of at least $2$,
so $m_{p_{j}} \ge m_{p_{\ell}}2^{\ell-j}$ for every $j \in \{1, \dots, \ell\}$.
We also have for every such $j$
\[
10 \le m_{p_j} \le k^{1/2}.
\]
Let $t$ be the total number of heavy steps. We have
\[
t 
\le \sum_{j=1}^{\ell} \left\lceil\frac{10k}{m_{p_{j}}^2}\right\rceil
\le \sum_{j=1}^{\ell} \frac{11k}{(m_{p_{\ell}}2^{\ell-j})^2}
\le \frac{11k}{m_{p_{\ell}}^2} \sum_{j=0}^{\infty} 2^{-2j}
\le \frac{11k}{m_{p_{\ell}}^2} \cdot \frac{4}{3}
\le \frac{k}{6}.
\]

Let $\Ucal'$ be the set of word matrices after phase $\ell$.
During the whole procedure, at most $nk$ light steps and $t \le k/6$ heavy steps were made.
We have
\begin{equation}
\label{eq:ubpr1}
|\Ucal'|
\ge |\Ucal| \left(1-\frac{1}{n}\right)^{nk} \left(\frac{1}{n}\right)^{t}
\ge |\Ucal| \frac{1}{e^{2k}} n^{-t}.
\end{equation}

The overall matrix $A_{\Ucal'}$ has at most $10 nk^{1/2}$ $1$-entries
and at least $t$ rows with a single $1$-entry.
The number of $k \times n$ word matrices contained in $A_{\Ucal'}$ is maximized
when each of its rows with at least $2$ $1$-entries has the same number of $1$-entries.
Thus,
\begin{equation}
\label{eq:ubpr2}
|\Ucal'| \le \left(\frac{10nk^{1/2}}{k-t} \right)^{k-t} 
\le n^{k-t} \left(\frac{12}{k^{1/2}} \right)^{k} \quad \text{since $t \le k/6$.}
\end{equation}

By combining Equations~\eqref{eq:ubpr1}~and~\eqref{eq:ubpr2}, we conclude that
\[
|\Ucal| 
\le n^{k-t} (12k^{-1/2})^{k} e^{2k} n^t
\le n^k k^{-k/2 + O(k / \log k)}.
\]
\end{proof}

\bibliographystyle{vlastni}

\begin{thebibliography}{1}
\expandafter\ifx\csname url\endcsname\relax
  \def\url#1{{\tt #1}}\fi
\expandafter\ifx\csname urlprefix\endcsname\relax\def\urlprefix{URL }\fi

\bibitem{BHP}
R.~C. Baker, G.~Harman and J.~Pintz, The difference between consecutive primes,
  {II}, {\em Proceedings of The London Mathematical Society\/} {\bf 83} (2001),
  532--562.

\bibitem{CK12}
J.~Cibulka and J.~Kyn\v{c}l, Tight bounds on the maximum size of a set of
  permutations with bounded {VC-dimension}, {\em J. Comb. Theory, Ser. A\/}
  {\bf 119}(7) (2012), 1461--1478.

\bibitem{Egorychev}
G.~P. Egorychev, Proof of the van der {Waerden} conjecture for permanents, {\em
  Siberian Mathematical Journal\/} {\bf 22} (1981), 854--859.

\bibitem{ErdosMoon}
P.~Erd\H{o}s and J.~M.~Moon, On the subgraphs of the complete bipartite graph. 
{\em Canad. Math. Bull.\/} {\bf 7} (1964), 35--39.

\bibitem{ErdosSimonovits}
P.~Erd\H{o}s and M.~Simonovits, Some extremal problems in graph theory,
{\em Col: Math. Soc. J. Bolyai\/} {\bf 4} (1969) 377--390.

\bibitem{Falikman}
D.~I. Falikman, Proof of the van der {Waerden} conjecture regarding the
  permanent of a doubly stochastic matrix, {\em Mathematical Notes\/} {\bf 29}
  (1981), 475--479.

\bibitem{Furedi96}
Z.~F{\"u}redi, On the number of edges of quadrilateral-free graphs, {\em J.
  Comb. Theory, Ser. B\/} {\bf 68}(1) (1996), 1--6.

\bibitem{FurediHajnal}
Z.~{F\"{u}redi} and P.~Hajnal, {Davenport--Schinzel theory of matrices}, {\em
  Discrete Mathematics\/} {\bf 103}(3) (1992), 233--251.

\bibitem{FKMS}
Z.~F{\"u}redi, I.~Kantor, A.~Monti and B.~Sinaimeri, On reverse-free codes and
  permutations, {\em SIAM J. Discrete Math.\/} {\bf 24}(3) (2010), 964--978.

\bibitem{MNBook}
J.~Matou\v{s}ek and J.~Ne\v{s}et\v{r}il, {\em Invitation to Discrete
  Mathematics\/}, Oxford University Press (1998), ISBN 0198502079.

\bibitem{Raz00}
R.~Raz, {VC-Dimension of Sets of Permutations}, {\em Combinatorica\/} {\bf
  20}(2) (2000), 241--255.

\bibitem{Sidorenko}
A.~Sidorenko, {Inequalities for functionals generated by bipartite graphs}, {\em Diskret. Mat.\/} {\bf 3} (3) (1991), 50--65.

\end{thebibliography}

\end{document}